\documentclass[11pt,a4paper]{article}
\usepackage{amssymb,comment,amsmath,amsthm}
\usepackage{graphicx}
\usepackage{color}
\usepackage{titling,titlesec}
\usepackage{url}
\usepackage{amsthm,lipsum}
\usepackage{geometry}
\usepackage{tikz}
\usepackage{hyperref}
\titleformat{\subsection}[runin]
{\normalfont\bfseries}{}{0em}{}[.]
\newcounter{foo}

\newfont{\blb}{msbm10 scaled\magstep1}
\newfont{\comp}{cmr12 scaled\magstep1}
\newfont{\compb}{cmr10 scaled\magstep2}
\newfont{\sbb}{cmssbx10 scaled\magstep3}
\newfont{\sbbb}{cmssbx10 scaled\magstep5}
\newfont{\sbs}{cmssbx10 scaled\magstep1}
\newtheorem{theorem}{Theorem}
\newtheorem{lemma}[subsection]{Lemma}
\newtheorem{claim}[subsection]{Claim}

\newtheorem{corollary}[subsection]{Corollary}
\newtheorem{proposition}{Proposition}
\newtheorem{definition}[foo]{Definition}

\newtheorem{example}{Example}

\newcommand{\im}{\mathrm{Im}}

\newcommand{\vep}{\varepsilon}
\newcommand{\Supp}{\mathrm{Supp}}
\newcommand{\Beta}{\mathcal B}
\allowdisplaybreaks
\title{A General Inequality for Walks in Graphs}
\author{
Chase Wilson\thanks{Department of Mathematics, University of California, San Diego, CA, 92093-0112 USA. E-mail: c7wilson@ucsd.edu}}
\begin{document}
\date{}
\setlength{\droptitle}{-2.5cm}
\maketitle
\vspace{-0.5in}

\begin{abstract}
Let $G$ be a graph and $w_k(G)$ denote the number of walks in $G$ of length $k$. For sequences $a_1, \cdots, a_n$ and $b_1, \cdots, b_n$ of non-negative integers such that $a_1 + \cdots + a_n = b_1 + \cdots + b_n$, we determine a simple necessary and sufficient condition on $a_1, \cdots, a_n, b_1, \cdots, b_n$ for the inequality
\[
w_{a_1}(G) \cdots w_{a_n}(G) \geq w_{b_1}(G) \cdots w_{b_n}(G)
\]
to hold for any graph $G$. 
\end{abstract}
\begin{section}{Introduction}
Let $G$ be a graph. A \textit{walk} on $G$ is a sequence of vertices $v_1, \cdots, v_n$ such that $\{v_i, v_{i + 1} \} \in E(G)$ for all $1 \leq i \leq n - 1$. We denote by $w_n(G)$ the number of walks on $G$ of length $n$. For our purposes, the length of a walk is the number of edges it contains.

\medskip

The study of counting walks in graphs has a long history. In 1965, Blakeley and Roy \cite{BlR} proved that $w_n(G) w_0(G)^{n - 1} \geq w_1(G)^n$ for all $n \geq 1$. Erd\H{o}s and Simonovits \cite{ES} proved that $w_t(G)^k \leq w_k(G)^t w_0(G)^{k - t}$ when both $k$ and $t$ are even and $k \geq t$. They conjectured that the same result holds when both $k$ and $t$ are odd. This was recently proved by Sa\v{g}lam \cite{MS}, and later Raymond and Blekherman \cite{BR} gave a substantially simpler proof. Further inequalities of this type have also been studied. Lagarias \cite{LM} showed that $w_{2a + b}(G) w_b(G) \leq w_0(G) w_{2a + 2b}(G)$ Dress and Gutman \cite{DG} showed that $w_{a + b}(G)^2 \leq w_{2a}(G) w_{2b}(G)$. Hemecke, Kosub, et al. \cite{HK} generalized these two results by proving the so called sandwich theorem which states $w_{2a + c} (G)w_{2(a + b) + c}(G) \leq w_{2a}(G) w_{2(a + b + c)}(G)$. They also proved $w_{2 \ell + p}(G)^k \leq w_{2 \ell}(G)^{k - 1} w_{2 \ell + pk}(G)$. 

\medskip

Observe each of the above inequalities are of the form $w_{a_1}(G) \cdots w_{a_n}(G) \geq w_{b_1}(G) \cdots w_{b_n}(G)$ where $a_1 + \cdots + a_n = b_1 + \cdots + b_n$ and the $a_i$ are, in general, more "spread out" than the $b_i$. A natural question then is "what are all inequalities of this type that are true?".

\medskip

We answer this question essentially in its entirety, proving the following.

\begin{definition} Let $a_1, \cdots, a_n, b_1, \cdots, b_n$ be non-negative integers such that $a_1 + \cdots + a_n = b_1 + \cdots + b_n$. Let $b^o_1 \geq \cdots \geq b^o_q$ be the odd integers (counted with multiplicity) among $b_1, \cdots, b_n$ and $b^e_1 \geq \cdots \geq b^e_{n - q}$ be the even integers (counted with multiplicity) among $b_1, \cdots, b_n$. We say that $(a_1, \cdots, a_n)$ \textit{majorizes} $(b_1, \cdots, b_n)$ \textit{with parity} if the following holds. For all $i \in [0, q]$ and $j \in [0, n - q]$ there are at least $(i + j)$ elements from $a_1, \cdots, a_n$ that are either even or greater than or equal to $b^o_i$ and, letting $a'_1, \cdots, a'_{i + j}$ be the $(i + j)$ largest such elements (counted with multiplicity),
\[
\left ( \sum_{k = 1}^i b^o_k \right ) + \left ( \sum_{k = 1}^j b^e_k \right ) \leq \sum_{k = 1}^{i + j} a'_{k}
\]
We use the convention that $b^o_0 = + \infty$.
\end{definition}
\begin{theorem} \label{main}
Let $a_1, \cdots, a_n, b_1, \cdots, b_n$ be non-negative integers such that $a_1 +\cdots + a_n = b_1 +\cdots + b_n$. Then the inequality $w_{a_1}(G) \cdots w_{a_n}(G) \geq w_{b_1}(G) \cdots w_{b_n}(G)$ holds for any graph $G$ iff $(a_1, \cdots, a_n)$ majorizes $(b_1, \cdots, b_n)$ with parity.
\end{theorem}
After this research, we became aware of similar research done by Raymond and Blekherman \cite{BR2}. They used Tropical Geometry to give a list of inequalties and showed that any true inequality of the above type could be deduced from one on their list. Their result is also even more general in the sense that it doesn't require $a_1 + \cdots + a_n = b_1 + \cdots + b_n$. However we still believe our result is a novel contribution since

\begin{enumerate}
\item It gives a explicit description of the true inequalities which is non-trivial to prove from just the generating list.
\item Our proof is much shorter and purely combinatorial.
\end{enumerate}

\medskip

Theorem \ref{main} holds for multi-graphs as well as graphs. By considering the adjacency matrix of $G$, and using rational approximation we can translate this into the language of symmetric matrices.
\begin{corollary} Let $a_1, \cdots, a_n, b_1, \cdots, b_n$ be non-negative integers such that $a_1 + \cdots + a_n = b_1 +\cdots + b_n$. Then
\[
\prod_{i = 1}^n \langle A^{a_i} \boldsymbol{1}, \boldsymbol{1} \rangle \geq \prod_{i = 1}^n \langle A^{b_i} \boldsymbol{1}, \boldsymbol{1} \rangle
\]
holds for any symmetric matrix $A$ with non-negative entries iff $(a_1, \cdots, a_n)$ majorizes $(b_1, \cdots, b_n)$ with parity. 
\end{corollary}

\begin{subsection}{Organization} In Section \ref{sec_p1}, we use entropy techniques inspired by Raymond and Blekherman \cite{BR} to prove the if direction in Theorem \ref{main} in the special case that $b_1 = b_2 = \cdots = b_n$. In Section \ref{sec3}, we prove the full if direction of Theorem \ref{main} by averaging many instances of the special case where $b_1 = b_2 = \cdots = b_n$. In Section \ref{sec4}, we prove the only if direction of Theorem \ref{main} by, given a pair of sequences $(a_1, \cdots, a_n)$, $(b_1, \cdots, b_n)$ for which $(a_1, \cdots, a_n)$ does not majorize $(b_1, \cdots, b_n)$ with parity, constructing a graph $G$ in which $w_{a_1}(G) \cdots w_{a_n}(G) < w_{b_1}(G) \cdots w_{b_n}(G)$. 

\end{subsection}
\begin{section}{Notation and Entropy} \label{sec2}

\begin{subsection}{Notation}
For a matrix $M$, we let $M_i$ denote the $i$th row of $M$. We let $\bf{1}$ be the all $1$ vector. For vectors $v, w$, we say $v \leq w$ if $v$ is less than or equal to $w$ point-wise. For a graph $G$ and non-negative integer $k$, we let $w_k(G)$ denote the number of walks of length $k$ in $G$. When the graph $G$ is obvious from context we just write $w_k$.For a vector $v$, we let $\dim(v)$ be the dimension of $v$. For a vector $v$, we let $|v| = \sum_{i = 1}^n |v_i|$ be its $L_1$ norm.

\medskip

We say that $f(x) \ll g(x)$ if $f = o(g)$. 
\end{subsection}
\begin{subsection}{Entropy} We use entropy in section \ref{sec_p1}. Let us recall the basics of entropy.
If $S$ is a finite set and $X, Y$ are $S$-valued random variable, then the {\em entropy of $X$} is given by
\[
	H[X] = -\sum_{y \in \im(X)} P(X = y) \log ( P(X = y) )
\]
and the {\em conditional entropy} $H[X | Y]$ is given by
\[
	H[X | Y] = E_{y \sim Y} [ H[(X | Y = y)]]
\]
The two main facts of entropy that we will make use of are as follows:
\begin{proposition}
Let $X,X_1,X_2,\dots,X_k$ be random variables where $X$ has values in $S$ and $X_i$ has values in $S_i$ for $1 \leq i \leq k$.
\begin{center}
\begin{tabular}{lp{5in}}
$1$. & {\rm [Uniform bound]} $H[X] \leq \log( |\im(X)| )$ with equality iff $X$ has uniform distribution.\\
$2$. & {\rm [The Chain Rule]} $(X_1,X_2, \cdots, X_k)$ is an $S_1 \times \cdots \times S_k$ random variable and $H[X_1, \cdots, X_k] = H[X_1] + H[X_2 | X_1] + \cdots + H[X_k | X_1, \cdots, X_{k - 1} ]$.
\end{tabular}
\end{center}
\end{proposition}
\end{subsection}
\end{section}
\end{section}

\begin{section}{A Simpler Inequality} \label{sec_p1}
In this section we prove the following Lemma which will be used to prove Theorem \ref{main}..
\begin{lemma} \label{simple} Let $a_1, \cdots, a_n, b$ be integers such that $a_1 + \cdots + a_n = bn$, then
\[
w_{a_1} \cdots w_{a_n} \geq w_b^n
\]
for any graph $G$ if either
\begin{itemize}
\item $b$ is odd and $a_i$ is even whenever $a_i < b$.
\item $b$ is even and $a_i$ is even for all $i$.
\end{itemize}
\end{lemma}
Note this is precisely the if direction in Theorem \ref{main} in the special case where $b_1 = \cdots = b_n$. We will focus on proving this inequality when the first bullet holds since the second will follow from the exact same argument.

\begin{definition} We say walks $p_1, \cdots p_n$ in a graph $G$ can be arranged into the walks $q_1,\cdots, q_k$ if for every ordered pair $(u, v)$ of adjacent vertices in $V(G)$, $(u, v)$ is crossed (with this orientation) the same number of times in $p_1, \cdots, p_n$ as in $q_1, \cdots, q_k$.
\end{definition}

In order to use entropy, we aim to show that given a walk $W$ of length $b$, there exists $n$ walks $\rho_1, \cdots, \rho_n$ of length $a_1,\cdots, a_n$ that can be arranged into $n$ walks $\gamma_1, \cdots, \gamma_n$ each of length $b$ such that for all $i$, $\gamma_i$ is either $W$ or the the reverse of $W$ and the start and endpoint of $\rho_i$ are endpoints of $W$.

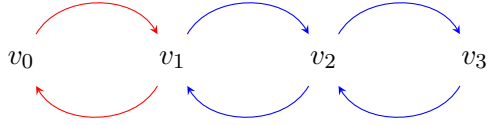
\[
\begin{tikzpicture}
\tikzset{
  every arrow/.style={line width=1.5pt},
  >=stealth
}
\node[circle] (A) at (0, 0){$v_0$};
\node[ circle] (B) at (2, 0) {$v_1$};
\node[ circle] (C) at (4, 0) {$v_2$};
\node[ circle] (D) at (6, 0) {$v_3$};

\node[align = flush center, text width = 8cm] (Label) at (3, -2)
{
A walk of length $2$ starting at $v_0$ (red) and a walk of length $4$ starting at $v_3$ (blue) arranged into two paths of length $3$ (top path and bottom path)
};

\path[red, ->]  (A) edge[out = 60, in = 120] (B);
\path[red, ->]  (B) edge[out = -120, in = -60] (A);

\path[blue, ->]  (D) edge[out = -120, in = -60] (C);
\path[blue, ->]  (C) edge[out = -120, in = -60] (B);

\path[blue, ->]  (B) edge[out = 60, in = 120] (C);
\path[blue, ->]  (C) edge[out = 60, in = 120] (D);

\end{tikzpicture}
\]
\begin{lemma}\label{decompose_lemma} Let $P_b = (0, \cdots, b)$ be the path of length $b$ and let $q_1, \cdots, q_{r}$ be paths in $P_b$ of the form $q_i = (u_i, v_i, u_i)$ for some vertices $v_i, u_i \in P_b$. Then if $s, t$ are positive integers such that $s + t = b + r$, there exists paths $p, \gamma$ such that
\begin{itemize}
\item $p$ has length $2s$, $\gamma$ has length $2t$.
\item $p$ and $\gamma$ can be arranged into $(0, \cdots, b), (b, \cdots, 0), q_1, \cdots, q_r$.
\item $p$ starts and ends on $0$ and $\gamma$ starts and ends on $b$
\end{itemize}
\end{lemma}
\begin{proof}
Let us assume that $q_1, \cdots, q_r$ are sorted by their start points so $u_1 \leq u_2 \leq \cdots \leq u_r$. We will start by constructing $p = (p_{0}, \cdots, p_{ 2s} )$ greedily. Conceptually, this means we walk along $P_b$, backtraking by $1$ to cover any $q_i$ we can, and then moving forward and finally returning to $0$ when we are out of length. Formally, we initiate $j = 1$ and do the following inductively: Suppose we have constructed $(p_{0}, \cdots, p_{ i - 1} )$. If $p_{ i - 1} = 2s - (i - 1)$, then let $(p_{i}, \cdots, p_{ 2s} ) = (2s - (i - 1) - 1, \cdots, 0)$ and let $M_1 = 2s - (i - 1)$. If that is not the case and $p_{ i - 1} = u_j$, then let $(p_{i}, p_{ i + 1}) = (v_j, u_j)$ and increment $j$ by $1$. Otherwise let $p_{ i} = p_{ i - 1} + 1$.

\medskip

If $j_1$ is the maximum $j$ gets incremented to, then we can arrange $p$ into $q_1, \cdots, q_{j_1 - 1}, (0,\cdots, M_1 - 1, M_1, M_1 - 1, \cdots, 0)$ when $M_1$ is the maximum $p_i$ is at the beginning of one of the steps.

\medskip

We then do the same greedy strategy to construct $\gamma$ but in reverse. We initiate $j = r$ and again do the following inductively: Suppose we have constructed $(\gamma_{0}, \cdots, \gamma_{ i - 1} )$. If $\gamma_{ i - 1} = b - ( 2t - (i - 1))$, then let $(\gamma_{i}, \cdots, \gamma_{ 2t} ) = ( b- (2t - (i - 1)) + 1, \cdots, b)$ and let $M_2 = b - 2t + i - 1$. Otherwise if $\gamma_{ i - 1} = u_j$, then let $(\gamma_{i}, \gamma_{ i + 1}) = (v_j, u_j)$ and decrement $j$ by $1$. Otherwise let $\gamma_{ i} = \gamma_{ i - 1} - 1$.

\medskip

If $j_2$ is the minimum $j$ gets decremented to, we can arrange $\gamma$ into $q_{j_2 + 1}, \cdots, q_r, (b, \cdots, M_2 + 1, M_2, M_2 + 1, \cdots b)$.

\medskip

Since for any non-negative integer $M$, $(0, \cdots,M - 1, M, M - 1, \cdots, 0), (b, \cdots, M + 1, M, M + 1, \cdots b)$ can be arranged into $(0, \cdots, b), (b, \cdots, 0)$, it suffices to show that $M_1 = M_2$ and $j_1 = j_2 + 1$. We have
\[
b + r = s + t = M_1 + (b - M_2) + (j_1 - 1) + (r - j_2)
\]
which implies
\[
(M_1 - M_2) = (j_2 + 1 - j_1)
\]
We also have $u_{j_1 - 1} \leq M_1 \leq u_{j_1}$ and $u_{j_2} \leq M_2 \leq u_{j_2 + 1}$ (Here we set $u_0 = 0$ and $u_{r + 1} = b$). Since $u_i$ is a non-decreasing sequence, this implies that $M_1 = M_2$ and $j_2 + 1 = j_1$
\end{proof}
\begin{lemma} \label{lemma_2} Let $(a_1, \cdots, a_n, r, b)$ be a squence of non-negative integers such that an even number of the $a_i$ are even, $a_1 + \cdots + a_n + 2r = n b$ either 1. $b$ is odd and all $a_i$ with $a_i < b$ are even or 2. $b$ is even and all $a_i$ are even . Then there exists paths $p_1, \cdots, p_n$, and $q_1, \cdots, q_r$ in $P_b = (0, \cdots, b)$ such that
\begin{enumerate}
\item $p_i$ has length $a_i$ for all $1 \leq i \leq n$.
\item $p_i$ starts and ends at an endpoint of $P_b$ for all $1 \leq i \leq n$
\item $q_i$ has length $2$ for all $i$
\item $q_i$ starts and ends at the same vertex for all $i$.
\item $p_1, \cdots, p_n, q_1, \cdots,q_r$ can be arranged into $P'_1, \cdots, P'_n$ where, for all $1 \leq i \leq n$, $P'_i$ is is either $(0,\cdots, b)$ or $(b, \cdots, 0)$.
\end{enumerate}
\end{lemma}
\begin{proof} We use induction. First suppose one of the $a_i$ is odd in which case $b$ is odd. Without loss of generality assume it is $a_n$. Then let us apply the inductive hypothesis with input $(a_1, \cdots, a_{n - 1}, r + (a_n - b)/2, b)$ to get paths $p_1, \cdots, p_{n - 1}$, $q_1, \cdots, q_{r + (a_n - b)/2}$. Let $p_n$ be a path in $P_b$ that can be arranged into $P_b, q_1, \cdots, q_{ (a_n - b)/2}$. Then it is not difficult to see that the paths $p_1, \cdots, p_n$ and $q_{(a_n - b)/2 + 1}, \cdots, q_{r + (a_n - b)/2}$ satisfy $1$-$5$.

\medskip

Now suppose that all the $a_i$ are even. Since $a_1 + \cdots + a_n + 2r = n b$, there must exist indices $1 \leq i < j \leq n$ such that $a_i + a_j + 2r \geq 2b$. Without loss of generality assume $i = n - 1$, $j = n$. Let us apply induction with $(a_1, \cdots a_{n - 2}, (a_i + a_j)/2 + r - b, b)$. Then we get paths $p_1, \cdots, p_{n - 2}$, and $q_1, \cdots, q_{ (a_i + a_j)/2 + r - b}$. If $a_i + a_j \geq 2b$, then by lemma \ref{decompose_lemma}, we can let $p_{n- 1}$ be a path that starts and ends at $0$, $p_{n}$ be a path that starts and ends at $b$ and $p_{n - 1}, p_n$ can be arranged into $(0, \cdots, b)$, $(b, \cdots, 0)$ and $q_{r + 1}, \cdots, q_{ (a_i + a_j - 2b)/2 + r }$. In this case the paths $p_1, \cdots, p_n$ and $q_1, \cdots, q_r$ will satisfy $1$-$5$. On the other hand if $a_i + a_j < 2b$, then we can let
\[
p_{n - 1} = (0, \cdots, a_{n - 1}/2, \cdots, 0)
\]
and
\[
p_n = (b, \cdots, b - a_n/2, \cdots, b )
\]
And for $1 \leq i \leq (2b - a_{n - 1} - a_n)/2$, we let
\[
q_{i + r + (a_{n - 1} + a_n - 2b)/2 } = ( a_{n - 1}/2 + i - 1, a_{n - 1}/2 + i, a_{n - 1}/2 + i - 1)
\]
In this case it is clear that $p_1, \cdots, p_n$ and $q_1, \cdots, q_{ r}$ will satisfy $1$-$5$.
\end{proof}

\begin{proof}[Proof of Lemma \ref{simple} ]
We will assume that an even number of the $a_i$ are even as otherwise we can show $w_{a_1}^2 \cdots w_{a_n}^2 \geq w_b^{2n}$ and take the square root of both sides. By Lemma \ref{lemma_2} with $r = 0$, there exists paths $p_1, \cdots, p_n$ in $P_b$ of lengths $a_1, \cdots, a_n$, each starting at an endpoint of $P_b$, that can be arranged into $P'_1, \cdots , P'_m, P''_1, \cdots, P''_{n - m}$ such that $P'_i$ is a copy of $P_b$ and $P''_i$ is a copy of $P_b$ reversed for all $i$. Let $\phi_i$ be the map $P_{a_i} \rightarrow P_b$ such defined by
\[
\phi_i ( P_{a_i} (k) ) = p_i (k)
\]
Let us consider the uniform distribution $X = (X_0, \cdots, X_b)$ on walks of length $b$ in $G$. We note that for all $1 \leq i\leq b - 1$, the distribution of $(X_{i +1}, \cdots, X_b)$ given $(X_0, \cdots, X_{i} )$ is the uniform distribution over walks of length $(b - i)$ starting at $X_i$. So $(X_{i +1}, \cdots, X_b)$ is conditionally independent to $(X_0, \cdots, X_{i - 1})$ given $X_i$. Therefore $H[ X_{i + 1} | X_i, \cdots, X_0 ] = H[X_{i + 1} | X_i ]$. Let $\theta^i_k$ be the unique integer such that $p_i(k) = P_b( \theta^i_k)$. Let $Y^i$ be the pullback of $X$ by $\phi_i$. In other words $Y^i$ is the distribution given by
\[
P \left ( Y^i = (y_0, \cdots, y_{a_i} ) \right ) = P( X_{\theta^i_0} = y_0) \prod_{k = 1}^{a_i} P(X_{\theta^i_k} = y_k | X_{\theta^i_{k - 1}} = y_{k - 1} )
\]
By conditional independence of $X$, we have $P( Y^i_k = y_1) = P(X_{\theta^i_k} = y_1)$ and $P( (Y^i_k, Y^i_{k - 1} ) = (y_1, y_2) ) = P( (X_{\theta^i_k}, X_{\theta^i_{k - 1}}) = (y_1, y_2) )$ for all $k$ and $y_1, y_2 \in V(G)$. Therefore by the chain rule,
\begin{align*}
H[Y^i] &= H[Y^i_0] + \sum_{k = 1}^{a_i} H[Y^i_k | Y^i_0, \cdots, Y^i_{k - 1} ]\\
& = H[X_{\theta^i_0}] + \sum_{k = 1}^{a_i} \sum_{v \in V(G)} P(X_{\theta^i_{k - 1}} = v) H[ X_{\theta^i_k} | X_{\theta^i_{k - 1}} = v ]\\
&= H[ X_{ \theta^i_0} ] + \sum_{k = 1}^{a_i} H [ X_{ \theta^i_k } | X_{ \theta^i_{k - 1}} ]
\end{align*}
Let $s$ be the number of $i$ such that $p_i(0) = P_b(0)$. Then $n - s$ is the number of $i$ such that $p_i(0) = P_b(b)$. Since paths are symmetric and $X$ is the uniform distribution, we note that $X_0$ has the same distribution as $X_b$. Using this we obtain,
\begin{align*}
\sum_{i = 1}^n \log (w_{a_i}) & \geq \sum_{i = 1}^n H[Y^i] \\
& = s H[X_0] + (n - s) H[X_b] + m \left ( \sum_{i = 1}^b H[X_i | X_{i - 1} ] \right ) + (n - m) \left ( \sum_{i = 1}^b H[X_{b - i} | X_{b - i + 1} ] \right ) \\
& = m \left ( H[X_0] + \sum_{i = 1}^b H[X_i | X_{i - 1} ] \right ) + (n - m) \left ( H[X_b] + \sum_{i = 1}^b H[X_{b - i} | X_{b - i + 1} ] \right ) \\
& = m H(X_0, \cdots, X_b) + (n - m) H(X_b, \cdots, X_0)\\
& = n \log w_b
\end{align*}
Taking exponents we obtain $ w_{a_1}\cdots w_{a_n} \geq w_{b}^n$.
\end{proof}

\end{section}

\begin{section}{An Averaging Argument} \label{sec3}
In this section we prove Theorem \ref{main} by averaging instances of Lemma \ref{simple}. That is we aim to prove our main inequality of the form $w_{a_1} \cdots w_{a_n} \geq w_{b_1} \cdots w_{b_n}$ by multiplying together instances of Lemma \ref{simple}. To this end we introduce the following definition.
\begin{definition} If $a, b \in \mathbb{Z}^n$, an \textit{allocation} of $a$ onto $b$ is a matrix $M \in \mathbb{R}^{n \times n}$ such that
\begin{itemize}
\item $ M_{i, j} \geq 0$ for all $1 \leq i, j \leq n$
\item $ M \boldsymbol{1} = \boldsymbol{1}$
\item $ \boldsymbol{1}^T M = \boldsymbol{1}^T$
\item $ M a = b$.
\item If $a_i$ is odd and either $b_j$ is even or $b_j > a_i$, then $M_{j, i} = 0$.
\end{itemize}
\end{definition}
We will say that the non-zero components of $M_i$ are the "parts of $a$ being allocated onto $b_i$".
\begin{definition} A \textit{rational allocation} of $a$ onto $b$ is an allocation such that $M \in \mathbb{Q}^{n \times n}$.
\end{definition}
The notion of a rational allocation is exactly what we need to be able to average instances of Lemma \ref{simple}
\begin{lemma} \label{rational_allocation_implies_theorem} If $a_1, \cdots, a_n, b_1, \cdots, b_n$ are non-negative integers and there exists a rational allocation of $a_1, \cdots, a_n$ onto $b_1, \cdots, b_n$, then $w_{a_1}, \cdots w_{a_n} \geq w_{b_1} \cdots w_{b_n}$.
\end{lemma}
\begin{proof} Let $M$ be a rational allocation. Then there exists an integer $k$ such that $k M \in \mathbb{Z}^{n \times n}$. By Lemma \ref{simple} we have
\[
w_{a_1}^{ k M_{i, 1} } \cdots w_{a_n}^{k M_{i, n} } \geq w_{b_i}^k
\]
for all $1 \leq i \leq n$. By multiplying all these inequalities together and taking the $k$th root, we obtain
\[
w_{a_1} \cdots w_{a_n} \geq w_{b_1} \cdots w_{b_n}
\]
\end{proof}
What we would like to do is show that we can find a rational allocation whenever the criteria of Theorem \ref{main} is satisfied. To make things easier we reduce the task to finding an arbitrary allocation as shown in the following Lemma.
\begin{lemma} \label{allocation_implies_rational_allocation} If $a, b \in \mathbb{Z}^n$ and there exists an allocation of $a$ onto $b$, then there exists a rational allocation of $a$ onto $b$.
\end{lemma}
\begin{proof} An allocation exists iff a certain affine subspace $A$ in $\mathbb{R}^{n \times n}$ determined by linear equations with rational coefficients intersects the region $ R = \{x_1, \cdots, x_n : x_1, \cdots , x_n \geq 0 \}$. And a rational allocation exists iff $A \cap \mathbb{Q}^n$ intersects $R$. If $A$ intersects $R$ in exactly one position, then that position lies in $\mathbb{Q}^n$ since $R$ is determined by rational inequalities. Otherwise, $A$ intersects one of the faces $F_i = \{x : x_i = 0 \}$ and $R$ simultaneously in which case we can induct on the dimension.
\end{proof}
From here we give an algorithm to find an allocation. The algorithm has two main steps. In the first step, we deal with the allocation onto $b^e$, the even components of $b$. Recall that from Lemma \ref{simple} we are only allowed to use $a^e$ to allocate onto $b^e$. We do this with the \textit{greedy allocation strategy} where we process the components of $b^e$ one by one and attempt to allocate onto them parts of $a^e$ as close together as possible. For example if we are processing $b^e_1 = 6$ in the first step and $a^e = \langle 10, 8, 4, 2 \rangle$ we will prefer the frist row $M_1$ of our allocation $M$ to be $M_1 = \langle 0, 1/2, 1/2, 0 \rangle$ rather than $M_1 = \langle 1/2 , 0, 0, 1/2 \rangle$ because $4$ and $8$ are closer together than $10$ and $2$ even though both average to $6$.

\medskip

In the second part of our allocation strategy, we finish the allocation by allocating onto $b^o$. This is also done in a greedy manner but in a different sense. We process the allocations one by one in the order $b^o_1, b^o_2, \cdots, b^o_{\dim(b^o)}$. We prefer to use $a^o$ to $a^e$ at each step. That is, we prefer to have the non-zero components of $M_i$ correspond to components of $a^o$ rather than $a^e$. We prefer this to the point where we allow ourselves to temporarily "overallocate" where we let $\langle M_i, a \rangle \geq b_i$ and adjust later in subsequent steps. Notice the only point this will fail at step $i$ is that there is not enough non-allocated parts of $a^o$ that are at least $b^o_i$ (which recall are the only parts of $a^o$ we are allowed to use) and the remaining non-allocated parts of $a^e$ are too small. In this case we will take some parts from previously overallocated $b_j$.

\medskip

We now introduce some technical definitions which will allow us to describe our algorithm formally.
\begin{definition} Let $c = c_1 \geq \cdots \geq c_n$ be a sequence of real numbers and let $\omega$ be a sequence of weights where $0 \leq \omega_i \leq 1$ for all $i \in [1, n]$. Let $I_{c, \omega}(x)$ be equal to $i$ where $i$ is the minimum index for which that $\omega_1 + \cdots + \omega_{i - 1} \leq x \leq \omega_1 + \cdots + \omega_i$ and let
\[
J_{c, \omega}(x, y) = \int_{x}^y c_{I_{c, \omega}(t)} dt
\]
Let $H_{i, \omega} (x, y) = \mu \{ t \in [x, y] : I_{c, \omega} (t) = i \}$ where $\mu$ is the Lebesgue measure and let $H_{\omega}(x, y) = (H_{1, \omega}(x, y), \cdots, H_{n, \omega}(x, y))$. Note the equation
\[
\sum_{i = 1}^n c_i H_{i, \omega} (x, y) = J_{c, \omega} (x, y)
\]
\end{definition}
The following definition is used to describe the first part of our algorithm.
\begin{definition} If $k \leq n$, $a \in \mathbb{Z}^n$ and $b \in \mathbb{Z}^k$ a sub-allocation of $a$ onto $b$ is a matrix $M \in \mathbb{R}^{k \times n}$ such that
\begin{enumerate}
\item $M_{i, j} \geq 0$ for all $i \in [1, k], j \in [1, n]$
\item $M \boldsymbol{1} = 1$
\item $ \boldsymbol{1}^T M \leq \boldsymbol{1}^T$
\item $Ma = b$
\item If $a_i$ is odd and either $b_j$ is even or $b_j > a_i$, then $M_{j, i} = 0$.
\end{enumerate}
\end{definition}
\begin{definition} Given a vector $a \in \mathbb{Z}^n$, we let $a^e_1 \geq a^e_2 \geq \cdots a^e_k$ be the even integers among $(a_1, \cdots, a_n)$ counted with multiplicity and $a^o_1 \geq a^o_2 \geq \cdots \geq a^o_{n - k}$ be the odd integers among $(a_1, \cdots, a_n)$ counted with multiplicity. We let $a^e = (a^e_1, \cdots, a^e_k)$ and $a^o = (a^o_1, \cdots, a^o_{n - k} )$. Going forward we will assume without loss of generality that these vectors are already sorted so $a = (a^e_1, \cdots, a^e_k, a^o_1, \cdots, a^o_{n - k} )$.
\end{definition}
\begin{definition} Given $a, b \in \mathbb{Z}^n$, the \textit{greedy allocation strategy} is the sub-allocation of $a^e$ onto $b^e$ constructed by the following algorithm: Initially let $\omega^1 = (1, \cdots, 1) \in \mathbb{R}^{\dim(a^e)}$. Then for $i = 1 \cdots \dim(b^e)$, let $t$ be the minimum real number for which
\[
J_{a^e, \omega^i} (t, t + 1) = b^e_i
\]
if such a $t$ exists. If such a $t$ does not exist abort, otherwise let $M'_{i, j} = H_{j, \omega^i} (t, t + 1)$ and set $\omega^{i + 1}_j = \omega^i_j - H_{j, \omega^i}(t, t + 1)$.
\end{definition}
\begin{example}
Suppose that $a^e = (10, 6, 2, 0)$ and $b^e = (8, 4, 2)$. We can represent the pair $(a^e, \omega^i)$ by rectangles with numbers corresponding to $a^e$ and width corresponding to $\omega$. Initially $\omega^1 = (1, 1, 1, 1)$ so we have the following picture.
\[
\begin{tikzpicture}
\node[ draw, rectangle, minimum width=2cm, minimum height=2cm, inner sep=0pt] (A) at (0, 0) {10};
\node[ draw, rectangle, minimum width=2cm, minimum height=2cm, inner sep=0pt] (B) at (2, 0) {6};
\node[ draw, rectangle, minimum width=2cm, minimum height=2cm, inner sep=0pt] (C) at (4, 0) {2};
\node[ draw, rectangle, minimum width=2cm, minimum height=2cm, inner sep=0pt] (D) at (6, 0) {0};
\end{tikzpicture}
\]
In the first step we allocate half of $10$ and half of $6$ onto $b^e_1 = 8$ so $\omega^2 = (1/2, 1/2, 1, 1)$
\[
\begin{tikzpicture}

\node[ draw, rectangle, minimum width=1cm, minimum height=2cm, inner sep=0pt] (A2) at (-.5, 0) {10};
\node[ draw, rectangle, minimum width=1cm, minimum height=2cm, inner sep=0pt] (B2) at (.5 , 0) {6};
\node[ draw, rectangle, minimum width=2cm, minimum height=2cm, inner sep=0pt] (C2) at (2 , 0) {2};
\node[ draw, rectangle, minimum width=2cm, minimum height=2cm, inner sep=0pt] (D2) at (4, 0) {0};
\end{tikzpicture}
\]
In the second step we allocate the remaining part of $6$ and half of $2$ onto $b^e_2 = 4$ so $\omega^3 = (1/2, 0, 1/2, 1)$.
\[
\begin{tikzpicture}
\node[ draw, rectangle, minimum width=1cm, minimum height=2cm, inner sep=0pt] (A2) at (-.5, 0 ) {10};
\node[ draw, rectangle, minimum width=1cm, minimum height=2cm, inner sep=0pt] (B2) at (.5 , 0) {2};
\node[ draw, rectangle, minimum width=2cm, minimum height=2cm, inner sep=0pt] (C2) at (2 , 0) {0};
\end{tikzpicture}
\]
Finally in the last step we allocate $1/10$ of $10$, all the remaining part of $2$ and $2/5$ of $0$ onto $b^e_3 = 2$ so $\omega^4 = (2/5, 0, 0, 3/5)$.
\[
\begin{tikzpicture}
\node[ draw, rectangle, minimum width=.8cm, minimum height=2cm, inner sep=0pt] (A2) at (0, 0 ) {10};
\node[ draw, rectangle, minimum width=1.2cm, minimum height=2cm, inner sep=0pt] (B2) at (1 , 0) {0};

\end{tikzpicture}
\]

\end{example}

\begin{definition} The support of $M'_i$ denoted $\Supp(M'_i)$ is the set of all $1 \leq j \leq n$ such that $M'_{i, j} \neq 0$.
\end{definition}
\begin{lemma} {\label{support} } In our greedy allocation strategy, for every $i$,
\begin{enumerate}
\item The sequence $ \max \Supp(M'_i)$ is non-decreasing in $i$.
\item If $ \min \Supp(M'_i) < j < \max \Supp(M'_i)$, then $\omega^k_{j} = 0$ for $k > i$.
\end{enumerate}
\end{lemma}
\begin{proof} 2. follows directly from the construction of the greedy allocation strategy. For part 1., we note that if $i < j$ and $\max \Supp(M'_i) > \max \Supp(M'_j)$, then by part 2., we have either
\begin{itemize}
\item $\min \Supp(M'_i) < \max \Supp(M'_i)$ and $\min \Supp(M'_i) \geq \max \Supp(M'_j)$.
\item $\min \Supp(M'_i) = \max \Supp(M'_i)$ and $\min \Supp (M'_i) > \max \Supp(M'_j)$
\end{itemize} In either case this is a contradiction since $b^{e}_j \leq b^e_i$, $b^e_k$ is a convex combination of $c_r$ where $r$ ranges over the support of $M'_k$ for all $k$, $c_r$ is non-increasing, and $t$ is chosen to be minimal in the greedy allocation strategy.
\end{proof}
The following is a key Lemma that will be used repeatedly and demonstrates a lot of the power of the greedy allocation strategy.
\begin{lemma}\label{lemma_glue}
If our greedy allocation strategy completes at least $k$ steps, then for any integer $m \in [0, \omega^k_1 + \cdots + \omega^k_n]$,
\[
m = \omega^k_1 + \cdots + \omega^k_{I_{a^e, \omega^k} (m)}
\]
Moreover, for any $i < k$, the support of $M'_i$ lies completely in either $\{1, \cdots, I_{a^e, \omega^k} (m) \}$ or in $\{ I_{a^e, \omega^k} (m) + 1, \cdots , n \}$.
\end{lemma}
\begin{proof} We will induct on $k$ for which the base case is obvious. Let $(d^k_1, \cdots, d^k_{|\omega| - k + 1}) := ( I_{a^e, \omega^k}(1), \cdots, I_{a^e, \omega^k}(|\omega| - k + 1))$. Then we note, that for any integer $m \in [0, \omega^k_1 + \cdots + \omega^k_n - 1]$
\[
 \omega^{k - 1}_{1 + I_{a^e, \omega^{k - 1}} (m)}  + \cdots + \omega^{k - 1} _{ I_{a^e, \omega^{k - 1} }(m + 1)} = 1
\]
and therefore by the definition of the greedy allocation strategy there exists an index $j$ such that
\begin{enumerate}

\item For $i \neq j$, $ \{d^{k - 1}_i, d^{k - 1}_i + 1 \} \not \subseteq \Supp(M'_{j - 1})$ 
\item $  (d^k_1, \cdots, d^k_{|\omega| - k + 1} ) = (d^{k - 1}_1, \cdots, d^{k - 1}_{j - 1}, d^{k - 1}_{j + 1}, \cdots, d^{k - 1}-{|\omega| - k + 1})$.

\end{enumerate}

which by induction proves that for any $i < k$, the support of $M'_i$ lies completely in either $\{1, \cdots, I_{a^e, \omega^k} (m) \}$ or in $\{ I_{a^e, \omega^k} (m) + 1, \cdots , n \}$.

\medskip

We will now prove $m = \omega^k_1 + \cdots + \omega^k_{I_{a^e, \omega^k} (m)}$. For $j < k$, let 
\[
\vep_j = \begin{cases}  1 & \Supp(M'_j) \subseteq  \{1, \cdots, I_{a^e, \omega^k} (m) \}\\  0 & \Supp(M'_j) \subseteq \{ I_{a^e, \omega^k} (m) + 1, \cdots , n \}
\end{cases}
\]
Then,
\[
I_{a^e, \omega^k} (m) = \omega^k_1 + \cdots + \omega^k_{I_{a^e, \omega^k} (m)} + \sum_{j = 1}^{k - 1} \vep_j  \sum_{i = 1}^{n} M'_{j, i} = (\omega^k_1 + \cdots + \omega^k_{I_{a^e, \omega^k}(m)} - m) + m + \sum_{j = 1}^{k - 1} \vep_j
\]
Therefore $ \omega^k_1 + \cdots + \omega^k_{I_{a^e, \omega^k} (m)} - m $ is an integer. But $ 0 \leq \omega^k_1 + \cdots + \omega^k_{I_{a^e, \omega^k}(m)} - m < 1$ so $m = \omega^k_1 + \cdots + \omega^k_{I_{a^e, \omega^k}(m)}$.
\end{proof}
\begin{definition} We let $\gamma(i)$ be the largest integer such that $a^o_{\gamma(i)} \geq b^o_i$ and $0$ if no such integer exists.
\end{definition}

Let us restate the definition of majorizing with parity in a form that will be more useful.
\begin{definition} A set non-negative integers $\{ \beta(i, j) \}_{i \in [0, \dim(b^o)], j \in [0, \dim(b^e)]}$ is a majorizer of $b$ if for all $i, j$, $\beta(i, j) \leq \gamma(i)$, $0 \leq i + j - \beta(i,j) \leq \dim(a^e)$ and
\[
\left ( \sum_{k = 1}^i b^o_k \right ) + \left ( \sum_{k = 1}^j b^e_k \right ) \leq \left ( \sum_{k = 1}^{\beta(i, j)} a^o_k \right ) + \left (\sum_{k = 1}^{ i + j - \beta(i, j) } a^e_k \right )
\]
\end{definition}
Observe that majorizing with parity is equivalent to a majorizer existing.
\begin{lemma} \label{major_suffix} Suppose that there exists a majorizer $\beta(i, j)$. Then for any $0 \leq k \leq \dim(a^o) - 1$, $a^o_{\dim(a^o) - k} \geq b^o_{\dim(b^o) - k}$.
\end{lemma}
\begin{proof} We have $\beta(i, j) \geq i + j - \dim(a^e)$ and so
\[
\gamma( \dim(b^o) - k) \geq \beta( \dim(b^o) - k, \dim(b^e) ) \geq \dim(a^o) - k 
\]
and since $a^o_j$ is decreasing in $j$, $a^o_{\dim(a^o) - k} \geq a^o_{\gamma(\dim(b^o) - k)}  \geq b^o_{\dim(b^o) - k}$.
\end{proof}

\begin{definition}

The \textit{standard majorizer} is  the sequence $\{ \Beta(i) \}_{i \in [0, \dim(b^o)]}$ constructed inductively by $\Beta(0) = 0$ and for $1 \leq i \leq \dim(b^o)$,
\[
\Beta(i) = \begin{cases}  \Beta(i -1) + 1 & \Beta(i - 1) < \gamma(i) \\ \Beta(i - 1) & \Beta(i - 1) = \gamma(i) \end{cases}
\]

\end{definition}

\begin{lemma} \label{increasing}
Suppose there exists a majorizer $\beta(i,j)$. Then 
\[
\beta'(i, j) := \Beta(i)
\]
is also a majorizer.

\end{lemma}
\begin{proof}

Suppose towards contradiction that $\beta'$ is not a majorizer. Then let $i_0$ be the minimal non-negative integer such that there exists a $j_0$ where
\[
\sum_{i = 1}^{i_0} b^o_i  + \sum_{i = 1}^{j_0} b^e_i  >  \sum_{i  = 1}^{\beta'(i_0, j_0)} a^o_i  + \sum_{i  = 1}^{i_0 + j_0 - \beta'(i_0, j_0) } a^e_i
\]
Since, $\sum_{i = 1}^{j_0} a^e_i \geq \sum_{i  = 1}^{j_0} b^e_i$, we have $i_0 \geq 1$.
\begin{claim}
 We now claim that $\beta'(i_0 - 1, j_0) = \gamma(i_0)$. 
\end{claim}
\begin{proof}
Suppose towards contraduction that $\beta'(i_0 - 1, j_0) < \gamma(i_0)$. Then,
\begin{align*}
\sum_{i = 1}^{i_0} b^o_i  + \sum_{i = 1}^{j_0} b^e_i &= b^o_{i_0} +  \sum_{i = 1}^{i_0 - 1} b^o_i  + \sum_{i = 1}^{j_0} b^e_i \\
& \leq  a^o_{\beta'(i_0, j_0)}  +  \sum_{i  = 1}^{\beta'(i_0, j_0) - 1} a^o_i  + \sum_{i  = 1}^{i_0 + j_0 - \beta'(i_0, j_0) } a^e_i \\
&=  \sum_{i  = 1}^{\beta'(i_0, j_0) }a^o_i  + \sum_{i  = 1}^{i_0 + j_0 - \beta'(i_0, j_0) } a^e_i
\end{align*}
\end{proof}
Therefore if $ a^e_{i_0 + j_0 - \beta'(i_0, j_0)} \geq b^o_{i_0}$ we are done by induction, so suppose $ a^e_{i_0 + j_0 - \beta'(i_0, j_0)} < b^o_{i_0}$.  Then,
\begin{align*}
  \sum_{i = 1}^{i_0} b^o_i  + \sum_{i = 1}^{j_0} b^e_i & \leq  \sum_{i  = 1}^{\beta(i_0, j_0) } a^o_i  + \sum_{i  = 1}^{i_0 + j_0 - \beta(i_0, j_0)  } a^e_i \\
& =    \left (  \sum_{i  = 1}^{\beta'(i_0, j_0) } a^o_i  + \sum_{i  = 1}^{i_0 + j_0 - \beta'(i_0, j_0)  } a^e_i \right ) -    \left ( \sum_{i = \beta(i_0, j_0) + 1}^{\beta'(i_0, j_0) }   a^o_i -   a^e_{i_0 + j_0  - \beta'(i_0, j_0) - \beta(i_0, j_0) + i}  \right )\\
& \leq \left (  \sum_{i  = 1}^{\beta'(i_0, j_0) } a^o_i  + \sum_{i  = 1}^{i_0 + j_0 - \beta'(i_0, j_0)  } a^e_i \right ) -    \left ( \sum_{i = \beta(i_0, j_0) + 1}^{\beta'(i_0, j_0) }   a^o_i -   b^o_{i_0}  \right )\\
& \leq \left (  \sum_{i  = 1}^{\beta'(i_0, j_0) } a^o_i  + \sum_{i  = 1}^{i_0 + j_0 - \beta'(i_0, j_0)  } a^e_i \right ) 
\end{align*}

\end{proof}
\begin{lemma} \label{no_abort} Suppose that there exists a majorizer $\beta(i, j)$, then the greedy allocation strategy does not abort.
\end{lemma}
\begin{proof} Recall that the greedy allocation strategy aborts iff there is some step $k$ in which we cannot find a $t$ for which $J_{a^e, \omega ^k } (t, t + 1) = b^e_k$. By the intermediate value theorem, this only happens in one of the following three cases
\begin{enumerate}
\item $|\omega^k|< 1$
\item $J_{a^e, \omega^k} (0, 1) < b^e_k$
\item $J_{a^e, \omega^k } ( |\omega^k| - 1, |\omega^k| ) > b^e_k$
\end{enumerate}
We will show none of these can happen. First we show $|\omega^k | \geq 1$. We have $|\omega^i| = \dim(a^e) - i  + 1\geq \dim(a^e) - \dim(b^e) + 1$ and plugging in $s = 0, j = \dim(b^e)$ to the condition that $s + j \leq \gamma(s) + \dim(a^e)$ gives $\dim(b^e) \leq \dim(a^e)$.

\medskip

Next let us show that $J_{a^e, \omega^k} (0, 1) \geq b^e_k$. By lemma \ref{lemma_glue}, for every $j \leq k - 1$, the support of $M'_j$ is completely contained in either $\{1, \cdots , I_{a^e, \omega^k} ( 1) \}$ or $\{ I_{a^e, \omega^k} (1) +1, \cdots, \dim(a^e) \}$. If the later is the case then
\[
J_{a^e, \omega ^k} (0, 1) \geq \sum_{i} c_i M'_{j,i} = b^e_j \geq b^e_k
\]
If the former is the case for every $1\leq j \leq k - 1$, then
\[
J_{a^e, \omega^k} (0, 1) + \sum_{i = 1}^{k - 1} b^e_i = J_{a^e, \omega^k} (0, 1) + \sum_{j = 1}^{k - 1} \sum_i c_i M'_{j, i} = J_{a^e, \omega^1} (0, k) = \sum_{i = 1}^k a^e_i
\]
Rearranging give
\[
J_{a^e, \omega^k}(0, 1) = \sum_{i = 1}^{k} a^e_i - \sum_{i = 1}^{k - 1} b^e_i \geq b^e_k
\]
Finally we will show $J_{a^e, \omega^k } ( |\omega^k| - 1, |\omega^k| ) \leq b^e_k$. Note that $s := \dim(b^o) - \dim(a^o) = \dim(a^e) - \dim(b^e) \geq 0$. Let $r$ be the number of $M'_i$ whose support is contained in $\{1, \cdots, I_{a^e, \omega^k} (s) \}$. Let $t$ be the maximum non-negative integer such that $\beta(t, r) = (t + r) - (r + s) = t - s$. By Lemma \ref{increasing} we may assume $\beta(i, j) \leq \beta(i + 1, j) \leq \beta(i, j) + 1$ and $\beta(i, j) \leq \beta(i, j + 1)$. Note that $\beta(\dim(b^o), r) \leq \dim(a^o) =\dim(b^o) - s$ and $\beta(0, r) = 0 \geq -s$ so such a $t$ must exist by the discrete intermediate value theorem and Lemma \ref{increasing}. By lemma \ref{lemma_glue},
\begin{align}
\sum_{i = 1}^{t} b^o_i + \sum_{i = 1}^{r} b^e_i & \leq \sum_{i = 1}^{ \beta(t, r) } a^o_i + \sum_{i = 1}^{s + r} a^e_i\\
& = \sum_{i = 1}^{\beta(t, r)} a^o_i + J_{a^e, \omega^k} (0, s) + \sum_{i = 1}^r b^e_i \\
&= \sum_{i = 1}^{\beta(t, r)} a^o_i + \sum_{i = 1}^{\dim(a^e)} a^e_i - J_{a^e, \omega^k} (s, |\omega^k| ) - \sum_{i = r + 1}^{k - 1} b^e_i\\
& = \sum_{i = 1}^{\dim(b^o)} b^o_i + \sum_{i = 1}^{\dim(b^e)} b^e_i - \sum_{i = \beta(t, r) + 1}^{\dim(a^o)} a^o_i - J_{a^e, \omega^k} (s, |\omega^k| ) - \sum_{i = r + 1}^{k - 1} b^e_i
\end{align}
Rearranging gives
\[
J_{a^e, \omega^k} (s, |\omega^k| ) \leq \left ( \sum_{i = r + 1}^{\dim (b^e)} b^e_i - \sum_{i = r + 1}^{k - 1} b^e_i \right ) + \left ( \sum_{i = t + 1}^{\dim(b^o)} b^o_i - \sum_{i = \beta(t, r) + 1}^{\dim(a^o)} a^o_i \right ) \leq \sum_{i = k }^{\dim(b^e)} b^e_i
\]
In the last step we are using Lemma \ref{major_suffix} to deduce that $ \sum_{i = t + 1}^{\dim(b^o)} b^o_i - \sum_{i = \beta(t, r) + 1}^{\dim(a^o)} a^o_i \leq 0$. Using this we get,
\[
J_{a^e, \omega^k} (|\omega^k| - 1, |\omega^k| )\leq \frac{1}{\dim(b^e) - k + 1} J_{a^e, \omega^k} (s, |\omega^k| ) \leq \frac{1}{\dim(b^e) - k + 1} \sum_{i = k }^{\dim(b^e)} b^e_i \leq b^e_k
\]
\end{proof}
\definition For $a, b \in \mathbb{Z}^n$, we say that a suballocation $M'$ of $a^e$ onto $b^e$ extends to an allocation if there exists an allocation of $(a^e, a^o)$ onto $(b^e, b^o)$ of the form
\[
\left[
\begin{array}{c|c}
M' & 0 \\
\hline
\multicolumn{2}{c}{M^*}
\end{array}
\right]
\]
where $M^*$ is a $\dim(b^o) \times n$ matrix.
\begin{lemma}\label{extends} Let $a, b \in \mathbb{Z}^n$. Suppose that $M'$ is a suballocation of $a^e$ on $b^e$ and for $1 \leq i \leq \dim(a^e)$, let
\[
\omega_i = 1 - \sum_{j = 1}^{\dim(b^e)} M'_{j, i}
\]
Then $M'$ extends to an allocation if
\[
a^o_1 + \cdots + a^o_{\Beta(i)} + J_{a^e, \omega} (0, i - \Beta(i)) \geq b^o_1 + \cdots + b^o_i
\]
for all $1 \leq i \leq \dim(b^o)$
\end{lemma}
\begin{proof}
We will show by induction that for every $i = 0, \cdots, \dim(b^o)$, there exists a matrix $M^i \in \mathbb{R}^{i \times n}$ such that
\begin{enumerate}
\item $M^i \boldsymbol{1} = \boldsymbol{1}$
\item $\boldsymbol{1}^T M^i = ( H_{\omega} (0, i - \Beta(i) ),   \boldsymbol{1} \in  \mathbb{R}^{\Beta(i)},  \boldsymbol{0} \in \mathbb{R}^{\dim(a^o) - \Beta(i)} )$
\item $M^i a \geq  (b^o_1, \cdots, b^o_i)$
\item For all $1 \leq k \leq i$, if $j > \gamma(k)$, then $M^i_{k, j + \dim(a^e) } = 0$. 
\end{enumerate}
Before doing this induction lets verify that 
\[
M'' := \left[
\begin{array}{c|c}
M' & 0 \\
\hline
\multicolumn{2}{c}{M^{\dim(b^o)}}
\end{array}
\right]
\]
is an allocation of $a$ onto $b$. We have
\begin{enumerate}
\item $M'' \boldsymbol{1} =  (\boldsymbol{1} \in \mathbb{R}^{\dim(b^e)}, \boldsymbol{1} \in \mathbb{R}^{\dim(b^o)} ) = \boldsymbol{1}$
\item $\boldsymbol{1}^T M'' =  ( \boldsymbol 1 -    \omega+ H_{\omega}(0,  |\omega|),  \boldsymbol{1} \in \mathbb{R}^{\dim(a^o)} ) = \boldsymbol{1}$
\item $M'' a \geq b$
\item If $a_i$ is odd and either $b_j$ is even or $b_j > a_i$, then $M''_{j, i} = 0$.

\end{enumerate}
Therefore the only thing to verify is that $M'' a = b$. Let $s_i = (M'' a)_i - b_i$. By assumption $s_i  \geq 0$. 
\[
\sum_{i = 1}^n a_i = \boldsymbol{1}^T  M'' a  =   \sum_{i = 1}^n b_i   +  \sum_{i = 1}^n s_i
\]
This implies that $s_i = 0$ for all $i$ and so $M'' a = b$. 

\medskip

It remains to do the induction. The base case of $i = 0$ is trivial since $M^0$ being the empty matrix is satisfactory. Now suppose that the inductive hypothesis holds for $i - 1$. We spil into two cases. The first is that $\Beta(i) = \Beta(i - 1) + 1$. In that case we let
\[
M^i := \left[
\begin{matrix}
M^{i - 1} \\

\boldsymbol{0} \in \mathbb{R}^{ \dim(a^e) + \Beta(i - 1)} ,  1 ,  \boldsymbol{0} \in \mathbb{R}^{n - \dim(a^e) - \Beta(i)}

\end{matrix}
\right]
\]
in which case it is easy to see that $M^i$ satisfies the desired properties. 

\medskip

The second case is that $\Beta(i) = \Beta(i - 1)$. 
\begin{claim}
There exists real numbers $0 \leq \lambda_1, \cdots, \lambda_i \leq 1$ such that $\lambda_1 + \cdots + \lambda_i = 1$ and
\[
M^i := \left [
\begin{matrix}
\lambda_1 H_{\omega} (i - \Beta(i) - 1, i - \Beta(i))  &+ (1 - \lambda_1) M^{i - 1}_1 \\
\vdots & \vdots \\
\lambda_{i - 1} H_{\omega}(i - \Beta(i) - 1, i - \Beta(i)) &+ (1 - \lambda_{i - 1}) M^{i-1}_{i - 1} \\
\lambda_i  H_{\omega}(i - \Beta(i) - 1, i - \Beta(i))  &+ \sum_{j = 1}^{i - 1} \lambda_j  M^{i - 1}_j
\end{matrix}
\right ] a \geq (b^o_1, \cdots, b^o_i)
\]

\end{claim}
\begin{proof}
For $1 \leq j \leq i - 1$, let 
\[
\Lambda_j = \sup  \{  \lambda \in [0, 1] : \langle  \lambda H_{\omega}(i - \Beta(i) - 1, i - \Beta(i)) + (1 - \lambda) M^{i - 1}_j, a \rangle \geq b^o_j \}
\]
We consider two cases. The first is that $\sum_{j = 1}^{i - 1} \Lambda_j \geq 1$. In that case take $\lambda_i= 0$ and for each $1 \leq j \leq i - 1$ take $\lambda_j \in [0, \Lambda_j]$ such that $\lambda_1 + \cdots + \lambda_{i - 1} = 1$. In this case it is clear from the construction that for $1 \leq j \leq i - 1$, 
\[
\langle \lambda_{j} H_{\omega}(i - \Beta(i) - 1, i - \Beta(i)) + (1 - \lambda_{j}) M^{i - 1}_j , a \rangle \geq b^o_j
\]
and
\[
\langle \sum_{j = 1}^{i - 1}  \lambda_j M^{i-1}_j, a  \rangle \geq \sum_{j = 1}^{i - 1} \lambda_j b^o_j  \geq b^o_i
\]
and so $M^i a \geq (b^o_1, \cdots, b^o_i)$.

\medskip

The second case is that $\sum_{j = 1}^{i - 1} \Lambda_j < 1$. In this case for $1 \leq j \leq i - 1$, we let $\lambda_j = \Lambda_j$ and we let $\lambda_i = 1 - \sum_{j = 1}^{i - 1} \Lambda_j$. In this case we have
\[
 \left \langle \lambda_{j} H_{\omega}(i - \Beta(i) - 1, i - \Beta(i)) + (1 - \lambda_{j}) M^{i - 1}_j , a  \right \rangle = b^o_j
\]
and
\[
 \left \langle   \lambda_i H_{\omega}(i - \Beta(i) - 1, i - \Beta(i)) +  \sum_{j = 1}^{i - 1}  \lambda_j M^{i-1}_j, a  \right  \rangle = a^o_1 + \cdots + a^o_{\Beta(i)} + J_{a^e, \omega} (0, i - \Beta(i)) - (b^o_1 + \cdots + b^o_{i - 1} ) \geq  b^o_i
\]
and so $M^i a \geq (b^o_1, \cdots, b^o_i)$.

\end{proof}

The claim verifies property 3. Let us verify the other properties. First
\[
M^i \boldsymbol{1} =  \left (  \lambda_1 + (1 - \lambda_1), \cdots, \lambda_{i - 1} + (1 - \lambda_{i - 1}),  \sum_{k = 1}^{i} \lambda_k  \right ) = \boldsymbol{1}
\]
which verifies property 1. Next
\begin{align*}
\boldsymbol 1^T M^i &= \boldsymbol{1}^T M^{i - 1} +  \left (  H_{\omega} ( i - \Beta(i) - 1, i - \Beta(i)), \boldsymbol 0 \in \mathbb{R}^{\dim(a^o)}  \right ) \\
& =  \left ( H_{\omega} (0, i - \Beta(i) ),   \boldsymbol{1} \in  \mathbb{R}^{\Beta(i)},  \boldsymbol{0} \in \mathbb{R}^{\dim(a^o) - \Beta(i)}  \right )
\end{align*}
which verifies property 2. Finally property 4 holds by inductive hypothesis. 

\end{proof}

\begin{lemma} \label{lemma_step} Suppose that $a$ majorizes $b$ with parity, then there exists an integer $r$ such that $J_{a^e, \omega} (0, i - \Beta(i) ) = (a^e_1 + \cdots a^e_r) - (b^e_1 + \cdots + b^e_{r - i + \Beta(i)})$ and
\[
\left ( \sum_{k = 1}^{\Beta(i)} a^o_{k} \right ) + \left ( \sum_{k = 1}^{r} a^e_k \right )  \geq  \left ( \sum_{k = 1}^i b^o_k \right ) + \left ( \sum_{k = 1}^{r - i + \Beta(i) } b^e_k \right )
\]
\end{lemma}

\begin{proof} The fact that there exists an $r$ such that $J_{a^e, \omega} (0, i - \Beta(i) ) = (a^e_1 + \cdots a^e_r) - (b^e_1 + \cdots + b^e_{r - i + \Beta(i)})$ follows from Lemma \ref{lemma_glue} and the inequality follows from Lemma \ref{increasing}.

\end{proof}

\begin{lemma} \label{other_condition} If $a$ majorizes $b$ with parity then there exists an allocation of $a$ onto $b$.
\end{lemma}
\begin{proof}
By lemma \ref{no_abort}, we can perform the greedy allocation of $a^e$. By lemma \ref{extends} it suffices to show
\[
\left ( \sum_{k = 1}^{\Beta(i)} a^o_{k} \right ) + J_{a^e, \omega} (0, i - \Beta(i) ) \geq b^o_1 + \cdots + b^o_i
\]
for all $i$. By Lemma \ref{lemma_step}, there exists an integer $r$ such that
\[
J_{a^e, \omega} (0, i - \Beta(i) ) = (a^e_1 + \cdots a^e_r) - (b^e_1 + \cdots + b^e_{r - i + \Beta(i)})
\]
and
\[
\left ( \sum_{k = 1}^{\Beta(i)} a^o_{k} \right ) + J_{a^e, \omega} (0, i - \Beta(i) ) = \left ( \sum_{k = 1}^{\Beta(i)} a^o_k \right ) + \left ( \sum_{k = 1}^{r} a^e_k \right ) - \left ( \sum_{k = 1}^{r - i + \Beta(i) } b^e_k \right ) \geq \sum_{k = 1}^i b^o_k
\]
\end{proof}

We are finally ready to prove Theorem \ref{main}.

\begin{proof}[Proof of the if direction in Theorem \ref{main}] By Lemma \ref{rational_allocation_implies_theorem}, it suffices to show there exists a rational allocation of $a$ onto $b$. By Lemma \ref{allocation_implies_rational_allocation}, it suffices to show there exists an arbitrary allocation of $a$ onto $b$. By Lemma \ref{other_condition} this is true.
\end{proof}

\end{section}

\begin{section}{A Construction When $a$ Does Not Majorize $b$ With Parity} \label{sec4}
In this section we prove the only if direction of Theorem \ref{main}. So let $a_1, \cdots, a_n$ be a sequence that does not majorize $b_1, \cdots, b_n$ with parity. There exist indices $i_0, j_0$ such that either $\sum_{k = 1}^{i_0} b^o_k  + \sum_{k  = 1}^{j_0} b^e_k  >  \sum_{k  = 1}^{i_0 + j_0} a'_k$ or there are less than $i_0  + j_0$ elements in $a_1, \cdots, a_n$ that are either even or greater than or equal to $b^o_{i_0}$.

\subsection{The General Case} We will start with focusing on the main case where the former holds and $i_0 > 0$. The remaining cases are edge cases for which the same idea will work but we need to tweak the construction. 

\medskip

We may assume that $b^o_{i_0} > 1$ and
\begin{equation}
 A' : =  \min \{ a'_1, \cdots, a'_{i_0 + j_0} \}  <  \min \{ b^o_{i_0}, b^e_{j_0} \} \label{assumption}
\end{equation}
as otherwise we could find a smaller pair $(i_0, j_0)$.

\begin{subsection}{The Construction of $G$}
 We will construct a graph $G$ using a parameter $M$ which, for $M$ sufficiently large, will satisfy $w_{a_1}(G) \cdots w_{a_n}(G) < w_{b_1}(G) \cdots w_{b_n} (G)$. $G$ will have two components $H$ and $K$. $H$ will be constructed out of $b^o_{i_0} + 1$ parts, $H_0, \cdots, H_{b^o_{i_0}}$ where $|H_0| = |H_{b^o_{i_0}}| = M^3$ and $|H_1| = |H_2| = \cdots = |H_{b^o_{i_0} - 1}| = M$. For $i \in [1, b^o_{i_0} - 2]$, $H_i$ and $H_{i + 1}$ will form a complete bipartite graph, each vertex in $H_1$ will be adjacent to $M^2$ unique vertices in $H_0$ and each vertex in $H_{b^o_{i_0} - 1}$ will be adjacent to $M^2$ unique vertices in $H_{b^o_{i_0}}$.  $K$ will be any $M^{1 - \vep}$ regular graph on $M^{3 + A' \vep}$ vertices where $\vep < 1/(1 + b^o_{i_0} - A')$ is a sufficiently small constant.  
\end{subsection}
\[
\begin{tikzpicture}
\tikzset{
  every arrow/.style={line width=1.5pt},
  >=stealth
}
\node[circle, fill = black, draw] (A1) at (0, 0){};
\node[circle, fill = black, draw] (A2) at (0, 1){};
\node[circle, fill = black, draw] (A3) at (0, 2){};
\node[circle, fill = black, draw] (A4) at (0, 3){};
\node[circle, fill = black, draw] (A5) at (0, 4){};
\node[circle, fill = black, draw] (A6) at (0, 5){};
\node[circle, fill = black, draw] (A7) at (0, 6){};
\node[circle, fill = black, draw] (A8) at (0, 7){};

\node[circle, fill = black, draw] (B1) at (3, 2){};
\node[circle, fill = black, draw] (B2) at (3, 5){};

\node[circle, fill = black, draw] (C1) at (6, 2){};
\node[circle, fill = black, draw] (C2) at (6, 5){};

\node[circle, fill = black, draw] (D1) at (9, 0){};
\node[circle, fill = black, draw] (D2) at (9, 1){};
\node[circle, fill = black, draw] (D3) at (9, 2){};
\node[circle, fill = black, draw] (D4) at (9, 3){};
\node[circle, fill = black, draw] (D5) at (9, 4){};
\node[circle, fill = black, draw] (D6) at (9, 5){};
\node[circle, fill = black, draw] (D7) at (9, 6){};
\node[circle, fill = black, draw] (D8) at (9, 7){};

\node[circle] (H0) at (0, 8){$H_0$};
\node[circle] (H1) at (3, 8){$H_1$};
\node[circle] (H2) at (6, 8){$H_2$};
\node[circle] (H3) at (9, 8){$H_3$};

\path (A1) edge (B1);
\path (A2) edge (B1);
\path (A3) edge (B1);
\path (A4) edge (B1);

\path (A5) edge (B2);
\path (A6) edge (B2);
\path (A7) edge (B2);
\path (A8) edge (B2);

\path (B1) edge (C1);
\path (B1) edge (C2);
\path (B2) edge (C1);
\path (B2) edge (C2);

\path (D1) edge (C1);
\path (D2) edge (C1);
\path (D3) edge (C1);
\path (D4) edge (C1);

\path (D5) edge (C2);
\path (D6) edge (C2);
\path (D7) edge (C2);
\path (D8) edge (C2);

\node[align = flush center, text width = 8cm] (Label) at (4.5, -2)
{
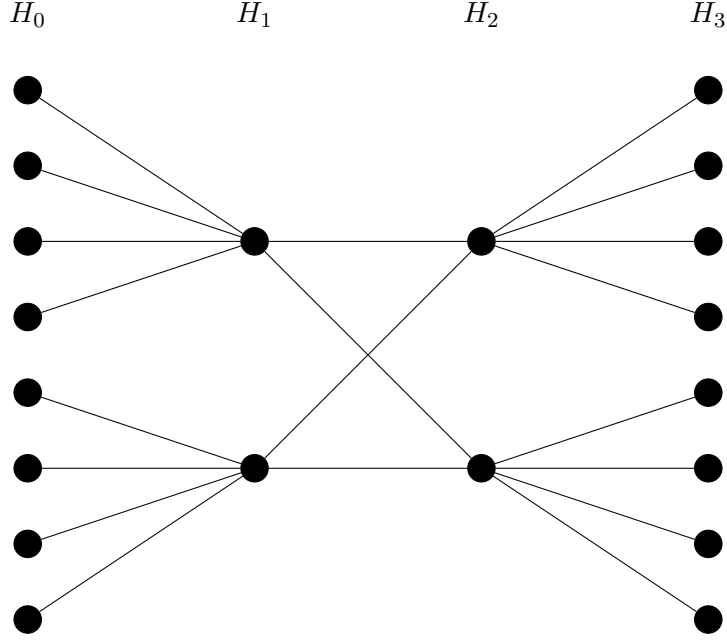
Picture of $H$ with $M = 2$ and $b^o_{i_0} = 3$.
};

\end{tikzpicture}
\]

\begin{lemma} \label{Hbound} Let $k$ be a non-negative integer and let 
\[
f(k) = \begin{cases}     M^{2 + k} &  k < b^o_{i_0}, k \equiv 1 \pmod 2 \\ M^{3 + k} & \text{else} \end{cases}
\]
then
\[
f(k)   \leq w_{k}(H)   \leq w_k(P_{b^o_{i_0} })  f(k)
\]

\end{lemma}
\begin{proof}
A walk of length $k$ in $H$ naturally projects down to a walk in $P_{b^o_{i_0}} = (0, \cdots, b^o_{i_0})$, by $\pi(p)_i$ is the index $j$ for which $p_i \in H_j$. Let us compute $|\pi^{-1}(p)|$. If $p_0 \in \{0, b^o_{i_0} \}$, then there are $M^3$ choices for $\pi^{-1}(p_0)$ and otherwise there are $M$ choices. Then for $i \geq 1$, there are $M^2$ choices for $\pi^{-1}(p_i)$ if $p_i \in \{0, b^o_{i_0}\}$, there is $1$ choice for $\pi^{-1}(p_i)$ if $p_{i - 1} \in \{0, b^o_{i_0}\}$, and there are $M$ choices otherwise. Putting this together we obtain, 
\[
|\pi^{-1} (p) | = M^{ k + 1 + \chi_{p_0  \in \{0, b^o_{i_0}\} } + \chi_{p_k \in \{0, b^o_{i_0}\} } }
\]  
Then the lemma follows from the fact that there exists a path $p$ of length $k$ such that $p_{0} \in \{0, b^o_{i_0} \}$ and $p_{k} \in \{ 0, b^o_{i_0} \}$ iff $k$ is even or $k$ is odd and greather than or equal to $b^o_{i_0}$.
\end{proof}

\begin{lemma} \label{weight_bound} Let $k$ be a non-negative integer. Then,
\[
w_k(G) = \begin{cases} \Theta(M^{3 + k} )  & k \equiv 0 \pmod 2,  k \geq A' \text{ OR }  k \equiv 1 \pmod 2,  k \geq b^o_{i_0} \\   
\Theta(M^{3 + A' \vep + (1 - \vep) k } )  &   \text{ else } 
\end{cases}
\]
\end{lemma}
\begin{proof} Using the $f(k)$ in Lemma \ref{Hbound},
\begin{align*}
w_k(G) & = \Theta ( \max \{ w_k(H), w_k(K) \} )\\
 & = \Theta ( \max \{  f(k),  M^{3 + A' \vep + (1 - \vep) k } \} )\\
& =  \begin{cases} \Theta(M^{3 + k} )  & k \equiv 0 \pmod 2,  k \geq A' \text{ OR }  k \equiv 1 \pmod 2,  k \geq b^o_{i_0} \\   
\Theta(M^{3 + A' \vep + (1 - \vep) k } )  &   \text{ else } 
\end{cases}
\end{align*}

\end{proof}
We can now prove the only if direction in Theorem \ref{main} in this general case.
\begin{proof} [Proof of the only if direction of Theorem \ref{main} in the general case ] Let $S := \sum_{k = 1}^n a_k = \sum_{k = 1}^n b_k$. Let $\hat{a}_1, \cdots, \hat{a}_{n - i_0 - j_0}$ be the elements in $a$ not appearing in $a'_1, \cdots, a'_{i_0 + j_0}$. By assumption \ref{assumption} each element in $\{\hat{a}_1, \cdots, \hat{a}_{n - i_0 - j_0} \}$ is either even and less than or equal to $A'$ or odd and less than $b^o_{i_0}$. By Lemma \ref{weight_bound},
\begin{align*}
w_{b_1}(G) \cdots w_{b_n}(G) & \geq \Theta \left ( \left ( \prod_{k = 1}^{i_0} M^{3 + b^o_k}  \right )  \left ( \prod_{k = 1}^{j_0} M^{3 + b^e_k} \right ) \left ( \prod_{k = i_0 + 1}^{\dim(b^o)}  M^{3 + A' \vep + (1 - \vep) b^o_k } \right ) \left ( \prod_{k = j_0 + 1}^{\dim(b^e) }M^{3 + A' \vep + (1 - \vep) b^e_k}  \right ) \right )\\
&=  \Theta \left ( M^{3n + (1 - \vep) S  +  (n - i_0 - j_0) A' \vep +  \left ( \sum_{k = 1}^{i_0}  b^o_k  + \sum_{k = 1}^{j_0} b^e_k \right ) \vep } \right ) \\
& \gg  \Theta \left ( M^{3n + (1 - \vep) S + (n - i_0 - j_0) A' \vep +  \left ( \sum_{k = 1}^{i_0 + j_0} a'_k \right ) \vep  } \right )\\
& = \Theta \left (  \left ( \prod_{k = 1}^{i_0 + j_0}  M^{3 + a'_{k} } \right ) \left ( \prod_{k = 1}^{n - i_0 - j_0} M^{3 + A' \vep + (1 - \vep) \hat{a}_k } \right ) \right ) \\
& = w_{a_1}(G) \cdots w_{a_n}(G)
\end{align*}
\end{proof}
In the above $\Theta$ is with respect to $M$.
\subsection{The Edge Cases} We will now describe how to change the construction of $G$ in the following three remaining cases
\begin{enumerate}
\item $i_0 = 0$
\item  There are less than $i_0  + j_0$ elements in $a_1, \cdots, a_n$ that are either even or greater than or equal to $b^o_{i_0}$. 
\end{enumerate}
If 1 holds, instead of having $b^o_{i_0} + 1$ parts, then we take $H$ to have at least $\max \{ b^o_1 , a^o_1 \} + 3$ parts (or an arbitrary number of parts if $a$ and $b$ have no odd components). If 2 holds then we let $G = H$ instead of $G = H \sqcup K$. Note that 1 and 2 may hold simultaneously in which case we both take $H$ to have at least $\max \{b^o_1, a^o_1 \} + 3$ parts and let $G = H$. In all these cases, the same computation that we did in the general case will apply.
\end{section}


\begin{thebibliography}{99}
\bibitem{BR} Blekherman, Raymond, Proof of the Erd\H{o}s-Simonovits conjecture on walks, Graphs and Combinatorics 39 (3) 53
\bibitem{BlR} George Blakley and Prabir Roy. H\"older type inequality for symmetric matrices with nonnegative entries. Proceedings of the American Mathematical Society, 16:1244-1245, 1965.
\bibitem{ES} P. Erd\H{o}s and M. Simonovits. Compactness results in extremal graph theory. Combinatorica, 2(3):275-283, 1982.
\bibitem{HK} R. Hemmecke, S. Kosub, E. mayr, H. T\"aubig, J. Weihmann. Inequalities for the Number of Walks in graphs. Algorithmica, 66, 804-828, (2013).
\bibitem{DG} A. Dress and I. Gutman. The number of walks in a graph. Applied Mathematics Letters, 16(5): 797-801, July 2003.
\bibitem{LM} J. Lagarias, J. Mazor, L. Shepp, B. McKay. An inequality for walks in a graph. IAM Review, 25(3):403, July 1983.
\bibitem{BR2} Blekherman, Grigoriy, and Annie Raymond. "A Path Forward: Tropicalization in Extremal Combinatorics." Advances in Mathematics, vol. 407, 2022, article 108561.
\bibitem{MS} Mert Sa\v{g}lam. Near log-convexity of measured heat in (discrete) time and consequences. In 59th IEEE Annual Symposium on Foundations of Computer Science. FOCS 2018, Paris, France, October 7-9, 2018, pages 967-978, 2018
\end{thebibliography}
\end{document}